\documentclass[english]{article}
\usepackage{amsmath}
\usepackage{amssymb}
\usepackage{esint}
\usepackage{enumitem}
\allowdisplaybreaks[4]
\makeatletter

\title{Newtheorem and theoremstyle test}

\usepackage[exercise]{amsthm}

\newtheorem{thm}{Theorem}[section]
\newtheorem{cor}[thm]{Corollary}

\newtheorem{lem}[thm]{Lemma}

\theoremstyle{remark}
\newtheorem*{rmk}{Remark}

\theoremstyle{plain}
\newtheorem{Def}{Definition}

\newtheoremstyle{note}
  {3pt}
  {3pt}
  {}
  {}
  {\itshape}
  {:}
  {.5em}
  {}

\theoremstyle{note}

\newtheoremstyle{citing}
  {3pt}
  {3pt}
  {\itshape}
  {}
  {\bfseries}
  {.}
  {.5em}
  {\thmnote{#3}}

\theoremstyle{citing}

\newtheoremstyle{break}
  {9pt}
  {9pt}
  {\itshape}
  {}
  {\bfseries}
  {.}
  {\newline}
  {}

\theoremstyle{break}

\theoremstyle{exercise}

\swapnumbers
\theoremstyle{plain}

\let\lvert=|\let\rvert=|

\usepackage{babel}

\usepackage{babel}

\usepackage{babel}

\usepackage{babel}

\makeatother

\usepackage{babel}
\begin{document}

\title{Isolated singularities of solutions of certain quasi-linear elliptic
inequalities}

\author{Shiguang Ma and Shengyang Zang}
\maketitle
\begin{abstract}
We provide a complete classification of the asymptotic behavior of
isolated singularities for solutions satisfying
\[
0\le-\Delta_{p}u(x)\le \tau u^{\frac{n(p-1)}{n-p}}(x),\,\,u(x)\ge0,\,\,1<p<n,\,\,n\ge2, 
\]where $u(x)\in C^2(B(0,1)\backslash\{0\})$ and $B(0,1)\subset \mathbb R^n$.
\end{abstract}

\section{Introduction}

The study of the behavior of isolated singularities in quasi-linear
equations is a fundamental topic in partial differential equations.
A key reference in this field is Serrin\textquoteright s work \cite{Serrin65},
which investigates the following class of equations

\begin{equation}
{\rm div}\mathcal{A}(x,u,u_{x})=\mathcal{B}(x,u,u_{x})\label{equation of divergence form}
\end{equation}
with the following assumptions
\begin{align}\label{structure condition 1}
|\mathcal{A}(x,u,q)| & \le a|q|^{p-1}+b|u|^{p-1}+e\nonumber \\
|\mathcal{B}(x,u,q)| & \le c|q|^{p-1}+d|u|^{p-1}+f\\
q\cdot\mathcal{A}(x,u,q) & \ge|q|^{p}-d|u|^{p}-g,\nonumber 
\end{align}
where $p\in(1,n]$ is a fixed number, and $0<a\in\mathbb{R}$, $b,c,d,e,f,g$
are measurable functions of $x$ contained in the respective Lebesgue
classes
\begin{equation}\label{structure condition 2}
    \begin{cases}
   & 
b,e\in L^{\frac{n}{p-1}};c\in L^{\frac{n}{1-\varepsilon}};d,f,g\in L^{\frac{n}{p-\varepsilon}},\,\,{\rm for}\,\,\,\,1<p<n,\\
&
b,e\in L^{\frac{n}{n-1-\varepsilon}};c\in L^{\frac{n}{1-\varepsilon}};d,f,g\in L^{\frac{n}{n-\varepsilon}},\,\,\,\,\,\,{\rm for}\,\,\,\,p=n,
\end{cases}
\end{equation}
with $\varepsilon$ being some positive constant.

Let $B(0,1)=\{x\in\mathbb{R}^n:|x|<1,n\ge2\}$. The main theorem in \cite{Serrin65} is stated as follows, which will be used frequently in the sequel.

\begin{thm}\label{Serrin's removable singularity}

Let $u$ be a continuous solution of (\ref{equation of divergence form})
in $B(0,1)\backslash\{0\}$. Suppose that $u\ge L$ for some constant
$L$ and that conditions (\ref{structure condition 1}) and (\ref{structure condition 2}) hold. Then either $u$ has a removable singularity at $0$, or there exist two positive constants $0<C'\le C''$ such that
\[
C'\le\frac{u(x)}{G_{p}(x)}\le C'' 
\]
as $x\to 0$, where 
\[
G_{p}(x)=\begin{cases}
|x|^{\frac{p-n}{p-1}}, & 1<p<n,\\
\log\frac{1}{|x|}, & p=n.
\end{cases}
\]

\end{thm}

Then it is a natural and interesting question whether the limit
\[
\lim_{x\to0}\frac{u(x)}{G_{p}(x)}
\]
exists. We notice one important special model of (\ref{equation of divergence form})
is the following $p$-Laplace type equation
\begin{equation}
-\Delta_{p}u=-{\rm div}(|\nabla u|^{p-2}\nabla u)=u^{\sigma},\label{p-Laplace equation}
\end{equation}
 which automatically fall into the categories above for $1\le\sigma<\frac{n(p-1)}{n-p}$ but nominally fails for $\sigma= \frac{n(p-1)}{n-p}$, due to the local integrability results for $p$-superharmonic functions(see Lemma\ref{integrability of u} below).

In the present paper, more generally, we study the asymptotic behavior of
solutions at $0\in\mathbb{R}^{n}$, of the following differential
inequality 
\begin{equation}
0\le-\Delta_{p}u\le \tau u^{\frac{n(p-1)}{n-p}},u\ge0,1<p<n,\label{main differential inequality}
\end{equation}
where $u\in C^{2}(B(0,1)\backslash\{0\})$, $\tau$ is a positive number. 

The equations of the form (\ref{p-Laplace equation}) were well studied in the literature. In the works \cite{Ni-Serrin 1, Ni-Serrin 2, Serrin-Zou},
two different critical powers were raised. $\sigma+1=q_{*}=\frac{p(n-1)}{n-p}$
is called the lower critical exponent and $\sigma+1=q^{*}=\frac{np}{n-p}$
is called the critical exponent for Sobolev embedding. Thus, the exponent in (\ref{main differential inequality}) we are considering corresponds to the lower critical case and below. 

When $p=2$, for the Laplace operator $\Delta$, the Lane-Emden equation
\begin{equation}
-\Delta u=u^{\sigma}\label{Lane-Emden equation}
\end{equation}
has a much longer history. For example, when
$\sigma=\frac{n+2}{n-2}$ is the critical exponent for Sobolev embedding,
the properties of its nonnegative and radial solutions were studied
by Emden \cite{Emden 1907} and Fowler \cite{Folwer1914,Fowler1931}.
Based on \cite{Lions,Aviles83}, Aviles \cite{Avilis87} proved that,
when $\sigma=\frac{n}{n-2}$ is the lower critical exponent, any 
nonnegative $C^2$ solution of (\ref{Lane-Emden equation}) (actually a more
general type equation was considered) in $B(0,1)\backslash\{0\}$ has only two possible asymptotic behaviors: either $u$
has a removable singularity at $0$ or 
\[
|x|^{n-2}(\log\frac{1}{|x|})^{\frac{n-2}{2}}u(x)\to(\frac{n-2}{\sqrt{2}})^{n-2}\,\,{\rm as}\,\,x\to0.
\]
In their famous paper \cite{Caffarelli-Gidas-Spruck}, by using the
moving plane method, Caffarelli, Gidas, and Spruck proved that, if
$\sigma\in[\frac{n}{n-2},\frac{n+2}{n-2}]$, 
any nonnegative $C^2$ solution
of (\ref{Lane-Emden equation}) in $B(0,1)\backslash\{0\}$ has
the form
\[
u(x)=(1+O(|x|))m(|x|),\,\,m(r)=\frac{1}{|\partial B(0,r)|}\int_{\partial B(0,r)}u,
\]
from which they further gave 
the characterization of the singular solution of (\ref{Lane-Emden equation}) for $\sigma\in[\frac{n}{n-2},\frac{n+2}{n-2}]$.

 As another direction to extend Aviles' result, Taliaferro \cite{Tailaferro01, Taliaferro06}
completely classified the asymptotic behavior of  nonnegative solutions
of 
\[
0\le-\Delta u\le \tau u^{\frac{n}{n-2}}
\]
 in $B(0,1)\backslash\{0\}$. We refer to \cite{Hopf31,Gidas-Spruck-81a,Gidas-Spruck-81b,Veron81,Simon83}
for other work about the Lane-Emden equation (\ref{Lane-Emden equation}).

Then we turn to the study of (\ref{p-Laplace equation}). Guedda and
Véron \cite{Guedda-Veron88} classified the isolated singularities of the nonnegative radial solutions of 
\[
-\Delta_{p}u=u^{\sigma},p-1<\sigma<\frac{np}{n-p}-1.
\]

G. Du and S. Zhou proved several results  in \cite{Du-Zhou2024}
about the isolated singularities of a more generalized equation 
\[
-\Delta_{p}u=|x|^{\alpha}u^{\sigma}.
\]
 They deal with subcritical case $p-1<\sigma<\frac{(n+\alpha)(p-1)}{n-p}$,
critical case $q=\frac{(n+\alpha)(p-1)}{n-p}$ and supercritical case
$\frac{(n+\alpha)(p-1)}{n-p}<q<\frac{(n+\alpha)p}{n-p}-1$. In all
cases, radial solutions were considered. When $p-1<\sigma<\frac{(n+\alpha)(p-1)}{n-p}$,
non-radial solutions were also considered and they proved that the
isolated singularity of a nonnegative solution is either removable or
behaves like a fundamental solution. 

A question related to the study of isolated singularities for (\ref{p-Laplace equation}) is the classification of entire solutions in $\mathbb{R}^n$, especially the critical exponent $\sigma=\frac{np}{n-p}-1$. We refer to \cite{Sciunzi2016} for $1<p<n$ with the additional  
assumption of finite energy and \cite{Ou2025} for $\frac{n+1}{3}<p<n$ but without any further restrictions.

Our aim is to generalize Taliaferro's works \cite{Tailaferro01,Taliaferro06}
to $p$-Laplace case. 

We define 
\[
\underline{u}(r)=\inf_{\partial B(0,r)}u(x),\,\,\bar{u}(r)=\sup_{\partial B(0,r)}u(x).
\]
Our main theorem is the following one.

\begin{thm}\label{main thm}

Let $u$ be a $C^{2}$ positive solution of (\ref{main differential inequality})
in $B(0,1)\backslash\{0\}$. Then 
\[
\lim_{x\to0}\frac{u(x)}{\underline{u}(|x|)}=1.
\]
Moreover, exactly one of  the following holds,
\begin{enumerate}
\item $u$ has $C^{1,\alpha}$-extension to $0$ for some $\alpha>0$;
\item 
\[
\lim_{x\to0}|x|^{\frac{n-p}{p-1}}u(x)=m\in(0,\infty);
\]
\item 
\[
\lim_{x\to0}|x|^{\frac{n-p}{p-1}}u(x)=0
\]
 and 
\[
\liminf_{x\to0}(\log\frac{1}{|x|})^{\frac{n-p}{p(p-1)}}|x|^{\frac{n-p}{p-1}}u(x)\ge (\frac{1}{\tau}(\frac{n-p}{p})(\frac{n-p}{p-1})^{p-1})^\frac{n-p}{p(p-1)}.
\]
\end{enumerate}
\end{thm}
\begin{rmk}
    When $m=0$, by Theorem \ref{Theorem 4 of LMQZ2023b} below one may obtain that $-\Delta_p u\in L^1_{loc}(B(0,1))$, and vice versa. 
\end{rmk}

Furthermore, we can rule out Item 3 of Theorem \ref{main thm} by slightly strengthening the upper bound control of $-\Delta_p u$.  We define 
\[\log_Q u(x)=\underbrace{\log \log \cdots \log}_{Q \text{ times}} u(x)\]
and then prove the following theorem.
\begin{thm} \label{strengthened thm}
    Let $u$ be a $C^2$ positive solution of 
\[0\le -\Delta_p u\le \frac{u^{\frac{n(p-1)}{n-p} }}{\log_1 u\cdot \log_2 u\cdots \log_{Q-1}u\cdot (\log_Q u)^\beta  }\] for some positive integer $Q>0$ and real number $\beta>1$. Then Item 3 of Theorem \ref{main thm} does not occur.
\end{thm}

We state our main approaches in proving Theorem \ref{main thm}  and \ref{strengthened thm}.

Section \ref{Sec2}  gives preliminaries. Notice that to prove the main theorem, the approach of \cite{Tailaferro01, Taliaferro06} does not work since the principle of superposition is no longer available for
$\Delta_{p},p\ne 2$.  Two important theorems we appeal to are Theorem \ref{KM theorem}, which controls positive $p$-superharmonic functions with Wolff potential, and Theorem \ref{Theorem 4 of LMQZ2023b}, which characterizes the singularities of a positive $p$-superharmonic function.

Section \ref{Sec3} through Section \ref{Sec6} are devoted to the proof of Theorem \ref{main thm}.

In Section \ref{Sec3}  we prove the uniform boundness of $u(x)/|x|^{\frac{p-n}{p-1}}$.  We localize the issue by choosing a sequence $x_{k}\to0$ such that $\lim\limits_{k\to\infty}u(x_{k})/|x_{k}|^\frac{p-n}{p-1}=\infty .$
And we do rescaling $v_{k}(\xi)=(\frac{|x_{k}|}{4})^{\frac{n-p}{p-1}}u(x_{k}+\frac{|x_{k}|}{4}\xi),\,\,{\rm for}\,\,\xi\in B(0,2)$
about $x_{k}$ and use Moser's iteration argument to find a contradiction. 

Then  in Section \ref{Sec4}   we further prove that it is not possible that $\lim\limits_{k\to\infty}u(x_{k})/|x_{k}|^\frac{p-n}{p-1}>m $, for any sequence $x_k\to0$.  

In Section \ref{Sec5},  we consider the case when $m=0$ and will prove that if $u(x)$ is bounded near a singularity, then the singularity is removable and if  $u(x)\to \infty $, the singularity is rotationally symmetric.  The key ingredient is
the Harnack inequality in \cite{Serrin 64}. Also, we need a Liouville-type theorem, Theorem \ref{Liouville type thm}.

In Section \ref{Sec6} we finally get the precise asymptotic lower bound of $u(x)$ when $m=0$ and $u(x)\to \infty$ as $x\to 0$. The main difficulty here is we have to find an alternative to the spherical average of $u$ that is widely used in the linear case $p=2$. Instead we consider 
\[\bar{u}(r)=\sup_{\partial B_r}u(x),\,\,\underline{u}(r)=\inf_{\partial B_r}u(x)\label{bar u and underline u}\]
   and resort to the ordinary differential inequalities that $\bar{u}$ and $\underline{u}$ satisfy. Since $\underline{u},\bar{u}$ may not be $C^2$ everywhere, one needs to treat the derivatives very carefully. This section is the most technical part of the paper.

   In Section \ref{Sec7}, we prove Theorem \ref{strengthened thm}, where we mainly follow the approach of \cite[Section 4]{Taliaferro06}.

\section{Preliminaries}\label{Sec2}

In this section, we give some basic knowledge about $p$-superharmonic
functions and Wolff potentials. The reader may refer to \cite{Lindqvist17,Kil-Maly-92,Kil-Maly-94}
for more details. 

\begin{Def}(\cite[Definition 2.5]{Lindqvist17}) Let $\Omega$ be
a domain in $\mathbb{R}^{n}$. We say that $u\in W_{loc}^{1,p}(\Omega)$
is a weak solution to 
\[
-\Delta_{p}u=0,x\in\Omega
\]
 if for each $\varphi\in C_{0}^{\infty}(\Omega)$, there holds
\[
\int_{\Omega}\langle|\nabla u|^{p-2}\nabla u,\nabla\varphi\rangle dx=0.
\]
If in addition, $u$ is continuous, then we say that $u$ is a $p$-harmonic
function. 

\end{Def}

\begin{Def}(\cite[Definition 5.1]{Lindqvist17}) A function $u:\Omega\to(-\infty,\infty]$
is called $p$-superharmonic in $\Omega$, if 
\begin{itemize}
\item $u$ is lower semi-continuous in $\Omega$; 
\item $u\not\equiv\infty$ in $\Omega$; 
\item for each domain $D\Subset\Omega$ the comparison principle holds:
if $h\in C(\bar{D})$ is $p$-harmonic in $D$ and $h|_{\partial D}\le u|_{\partial D},$
then $h\le u$ in $D$. 
\end{itemize}
\end{Def}

The following local integrability result for $p$-superharmonic functions is sharp, which explains why Theorem \ref{Serrin's removable singularity} cannot be applied directly to obtain the desired results in Section \ref{Sec5}. 

\begin{lem}(\cite[Theorem 5.13]{Lindqvist17})\label{integrability of u}\ If $u$ is $p$-superharmonic
in $\Omega,$ then 
\[
\int_{D}|u|^{q}dx<\infty
\]
where $D\Subset\Omega$ and $0\le q<\frac{n(p-1)}{n-p}$ in the case
$1<p\le n.$ In the case $p>n$ the function $u$ is continuous. 

\end{lem}

From \cite{Kil-Maly-92}, for any $p$-superharmonic function $u$,
there is a nonnegative Radon measure $\mu$ such that 
\begin{equation}
-\Delta_{p}u=\mu,\,\,x\in\Omega.\label{p-superharmonic}
\end{equation}

If $p=2$ the solution of (\ref{p-superharmonic}) can be represented
as 
\begin{equation}
u(x)=h(x)+\int_{\Omega}\Gamma(x-y)d\mu(y)\label{representation formula}
\end{equation}
 where $h(x)$ is a harmonic function and $\Gamma(x)$ is the fundamental
solution for $\Delta$ operator. 

When $p>1$ and $p\ne2$, no formula like (\ref{representation formula})
is available. However, the following important theorem can usually
be used as a substitute. First we define the $p$-Wolff potential
generated by a Radon measure $\mu$, for $1<p\le n$, as 
\[
W_{1,p}^{\mu}(x,r)=\int_{0}^{r}(\frac{\mu(B(x,t))}{t^{n-p}})^{\frac{1}{p-1}}\frac{dt}{t}.
\]

\begin{thm}\label{KM theorem}(\cite[Theorem 1.6]{Kil-Maly-94}) Let
$1<p\le n$. Suppose $u\ge0$ is a nonnegative $p$-superharmonic
function satisfying 
\[
-\Delta_{p}u=\mu,x\in B(x_{0},3r).
\]
 Then 
\[
c_{1}(n,p)W_{1,p}^{\mu}(x_{0},r)\le u(x_{0})\le c_{2}(n,p)(\inf_{B(x_{0},r)}u+W_{1,p}^{\mu}(x_{0},2r)).
\]

\end{thm}

The following theorem is important to us, since it characterizes the  properties of the singularity of  a nonnegative superharmonic function.

\begin{thm}\label{Theorem 4 of LMQZ2023b}(\cite[Theorem 4]{Liu-Ma-Qing-Zhong2}) Assume
$1<p\le n.$ Suppose that $u\ge0$ is a nonnegative $p$-superharmonic
function in $\Omega\subset\mathbb{R}^{n}$ satisfying 
\[
-\Delta_{p}u=\mu,
\]
where $\mu$ is a nonnegative Radon measure in $\Omega$. Then $\forall x_{0}\in\Omega$,
there is a subset $E$ which is $p$-thin for singular behavior at
$x_{0}$ such that 
\[
\lim_{x\to x_{0},x\notin E}\frac{u(x)}{G_{p}(x,x_{0})}=m=\begin{cases}
\frac{p-1}{n-p}(\frac{\mu(\{0\})}{|\mathbb{S}^{n-1}|})^{\frac{1}{p-1}} & {\rm when}\,p\in(1,n),\\
(\frac{\mu(\{0\})}{|\mathbb{S}^{n-1}|})^{\frac{1}{n-1}} & {\rm when}\,p=n,
\end{cases}
\]
where 
\[
G_{p}(x,x_{0})=\begin{cases}
|x-x_{0}|^{\frac{p-n}{p-1}} & {\rm when}\,p\in(1,n),\\
\log\frac{1}{|x-x_{0}|} & {\rm when}\,p=n.
\end{cases}
\]
Moreover, $u(x)\ge mG_{p}(x,x_{0})-c_{0}$ for some $c_{0}$ and all
$x$ in a neighborhood of $x_{0}$. 

\end{thm}

In the following sections, the notation  $m$  always refers to the one mentioned in the theorem above.
The readers may refer to \cite{Liu-Ma-Qing-Zhong, Liu-Ma-Qing-Zhong2}
for the definition of $p$-thinness for singular behavior (here we
note that ``$p$-thinness for singular behavior'' is a little different
from the traditional $p$-thinness, used in \cite{Kil-Maly-94}), which
 means roughly that a set is sparse near a certain point in a specific sense.

\section{Uniform boundedness of $u(x)/|x|^{\frac{p-n}{p-1}}$ }\label{Sec3}

To prove the main theorem, one basic step is to prove that $\frac{u(x)}{|x|^{\frac{p-n}{p-1}}}$
is uniformly bounded. Before this, we need two lemmas.

\begin{lem}\label{g+delta}(\cite[Theorem 1.1]{Bidaut Veron})

Under the assumption of Theorem \ref{main thm}, there are a function
$g\in L_{loc}^{1}(B(0,1))$ and a constant $\beta>0$ such that 
\[
-\Delta_{p}u=g+\beta\delta_0.
\]

\end{lem}

\begin{lem}\label{inf upper bound}

Suppose that $u\ge0$ satisfies 
\[
-\Delta_{p}u=\mu
\]
for a nonnegative Radon measure $\mu$. Then there is a constant $C>0$
such that 
\[
\inf_{B(x_{0},\frac{|x_{0}|}{2})}u(x)\le C|x_{0}|^{\frac{p-n}{p-1}}
\]
for each $x_{0}\in B(0,\frac{1}{2})\backslash\{0\}$. 

\end{lem}

\begin{proof}

For contradiction, we may assume that there is a sequence $x_{i}\in B(0,\frac{1}{4}),\,\,x_{i}\to0$
such that 
\[
\inf_{B(x_{i},\frac{|x_{i}|}{2})}u(x)\ge i|x_{i}|^{\frac{p-n}{p-1}}.
\]
 Setting $t_{i}=i|x_{i}|^{\frac{p-n}{p-1}},$ we have 
\begin{align*}
t_{i}^{\frac{n(p-1)}{n-p}}|x\in B(0,\frac{1}{2}):u(x)\ge t_{i}| & \ge t_{i}^{\frac{n(p-1)}{n-p}}|B(x_{i},\frac{|x_{i}|}{2})|\ge t_{i}^{\frac{n(p-1)}{n-p}}|x_{i}|^{n}C(n)=C(n)i^{\frac{n(p-1)}{n-p}}.
\end{align*}
 Then we must have $\|u\|_{L^{\frac{n(p-1)}{n-p},\infty}(B(0,\frac{1}{2}))}=\infty$.
However by \cite[Lemma 1.4]{Bidaut Veron}, 
\[
\|u\|_{L^{\frac{n(p-1)}{n-p},\infty}(B(0,\frac{1}{2}))}<\infty,
\]
 which is a contradiction. Then we proved the lemma. 

\end{proof}

\begin{lem}\label{uniformly bounded}

Under the assumption of Theorem \ref{main thm}, $\frac{u(x)}{|x|^{\frac{p-n}{p-1}}}$
is uniformly bounded.

\end{lem}

\begin{proof}

We argue by contradiction. Assume otherwise, there is a sequence $\{x_{k}\}$
inside the punctured ball such that
\[
\frac{u(x_{k})}{|x_{k}|^{\frac{p-n}{p-1}}}\to\infty,\,{\rm as}\,|x_{k}|\to0.
\]
 We consider the rescaled sequence of functions 
\begin{equation}
v_{k}(\xi)=(\frac{|x_{k}|}{4})^{\frac{n-p}{p-1}}u(x_{k}+\frac{|x_{k}|}{4}\xi),\,\,{\rm for}\,\,\xi\in B(0,2).\label{v_k  definition}
\end{equation}
Then $v_{k}(0)\to\infty$ as $k\to\infty$ while from Lemma \ref{inf upper bound}
we know 
\begin{equation}
\inf_{B(0,1)}v_{k}(\xi)\le C.\label{inf v_k  upper bound}
\end{equation}

We use $\Delta_{p}^{\xi}$ and $\Delta_{p}^{x}$ to denote $p$-Laplace
operator with respect to $(\xi_{1},\cdots,\xi_{n})$ and $(x_{1},\cdots,x_{n})$
coordinates, respectively. Define $g_k(\xi)=-(\frac{|x_{k}|}{4})^{n}\Delta_{p}^{x}u(x_{k}+\frac{|x_{k}|}{4}\xi)$, then we have 
\begin{align}
-\Delta_{p}^{\xi}v_{k}(\xi)&=  g_{k}(\xi)\le \tau v_{k}^{\frac{n(p-1)}{n-p}}(\xi),\,{\rm for}\,\xi\in B(0,2),\label{g_k definition}\\
  \int_{B(0,2)}g_{k}(\xi)d\xi&=\int_{B(x_{k},\frac{|x_{k}|}{2})}g(x)dx\to0\,\,{\rm as}\,k\to\infty.\nonumber 
\end{align}
The reason why $\int_{B(x_{k},\frac{|x_{k}|}{2})}g(x)dx\to0$ as $k\to\infty$
is due to Lemma \ref{g+delta}. We use $\mu_{f}$ to represent the
Radon measure defined by the $L^{1}$ function $f$, i.e. 
\[
\mu_{f}(\Omega)=\int_{\Omega}f(x)dx
\]
 holds for any bounded domain $\Omega$. 

We denote 
\[
\lambda_{k}=|x_{k}|^{\frac{p-n}{p-1}}\to\infty,\,{\rm as}\,k\to\infty.
\]

From Theorem \ref{KM theorem}, we know that 
\[
v_{k}(0)\le c_{2}(n,p)(\inf_{B(0,1)}v_{k}+W_{1,p}^{\mu_{g_{k}}}(0,2)).
\]
From (\ref{inf v_k  upper bound}), we know 
\[
W_{1,p}^{\mu_{g_{k}}}(0,2)\to\infty,\,{\rm as}\,k\to\infty.
\]

On the other hand, we will prove there exists $M>0$ such that 
\[
W_{1,p}^{\mu_{g_{_{k}}}}(0,2)\le M
\]
 and then we get a contradiction. We do this by showing that there
exist positive constants $C$ and $\epsilon$ such that 
\[
\mu_{g_{k}}(B(0,t))\le Ct^{n-p+\epsilon}.
\]
 Once we get this, we know that 
\begin{align*}
W_{1,p}^{\mu_{k}}(0,2) & =\int_{0}^{2}(\frac{\mu_{k}(B(0,t))}{t^{n-p}})^{\frac{1}{p-1}}\frac{dt}{t}\le C^{\frac{1}{p-1}}\int_{0}^{2}t^{\frac{\epsilon}{p-1}}\frac{dt}{t}\le C^{\frac{1}{p-1}}\frac{p-1}{\epsilon}2^{\frac{\epsilon}{p-1}}.
\end{align*}
Following \cite{Brezis}, we choose $R$ and $\lambda$ such that
$0<R\le\frac{1}{2}$ and $\frac{n-p+1}{n-p}(p-1)<\lambda<C(n,p)$.
Let $\phi\in C_{0}^{\infty}(B(0,2R))$ satisfy $|\nabla\phi|\le\frac{C}{R}$
and $\phi\equiv1,$ in $B(0,R)$. In the proof below, all derivatives
are taken with respect to $\xi$. 

Then, since 
\begin{align*}
\int_{B(0,2R)}|\nabla v_{k}|^{p}v_{k}^{\lambda-p}\phi^{p} & =(\frac{p}{\lambda})^{p}\int_{B(0,2R)}|\nabla v_{k}^{\frac{\lambda}{p}}|^{p}\phi^{p}\\
 & \ge(\frac{p}{\lambda})^{p}\int_{B(0,2R)}(2^{1-p}|\nabla(\phi v_{k}^{\frac{\lambda}{p}})|^{p}-(v_{k}^{\frac{\lambda}{p}})^{p}|\nabla\phi|^{p})
\end{align*}
 and by Young's inequality with $\varepsilon=\frac{1}{2(n-p)}$
\begin{align*}
 & \int_{B(0,2R)}|\nabla v_{k}|^{p-2}v_{k}^{\lambda-(p-1)}\phi^{p-1}\nabla v_{k}\nabla\phi\\
\le & \frac{\varepsilon(p-1)}{p}(\int_{B(0,2R)}|\nabla v_{k}|^{p-1}v_{k}^{\frac{(\lambda-p)(p-1)}{p}}\phi^{p-1})^{\frac{p}{p-1}}\\
 & +C(n,p)\int_{B(0,2R)}(v_{k}^{\frac{\lambda}{p}}|\nabla\phi|)^{p},
\end{align*}

we have 
\begin{align}
 & \int_{B(0,2R)}(-\Delta_{p}v_{k})v_{k}^{\lambda-(p-1)}\phi^{p}\nonumber \\
= & \int_{B(0,2R)}|\nabla v_{k}|^{p-2}\nabla v_{k}\nabla(v_{k}^{\lambda-(p-1)}\phi^{p})\nonumber\\
= & (\lambda-p+1)\int_{B(0,2R)}|\nabla v_{k}|^{p}v_{k}^{\lambda-p}\phi^{p}+p\int_{B(0,2R)}|\nabla v_{k}|^{p-2}v_{k}^{\lambda-(p-1)}\phi^{p-1}\nabla v_{k}\nabla\phi\nonumber \\
\ge & (\lambda-(1+\varepsilon)(p-1))\int_{B(0,2R)}|\nabla v_{k}|^{p}v_{k}^{\lambda-p}\phi^{p}-C(n,p)\int_{B(0,2R)}(v_{k}^{\frac{\lambda}{p}}|\nabla\phi|)^{p}\nonumber \\
\ge & (\lambda-(1+\varepsilon)(p-1))(\frac{p}{\lambda})^{p}\int_{B(0,2R)}(2^{1-p}|\nabla(\phi v_{k}^{\frac{\lambda}{p}})|^{p}-(v_{k}^{\frac{\lambda}{p}})^{p}|\nabla\phi|^{p})\nonumber \\
 & -C(n,p)\int_{B(0,2R)}(v_{k}^{\frac{\lambda}{p}}|\nabla\phi|)^{p}\nonumber \\
 \ge & C_{1}(n,p)\int_{B(0,2R)}|\nabla(\phi v_{k}^{\frac{\lambda}{p}})|^{p}-C_{2}(n,p)\int_{B(0,2R)}(v_{k}^{\frac{\lambda}{p}}|\nabla\phi|)^{p}\nonumber\\
\ge & C_{1}(n,p)||v_k^\lambda \phi^p||_{L^{\frac{n}{n-p}}(B(0,2R))}-C_{2}(n,p)\int_{B(0,2R)}(v_{k}^{\frac{\lambda}{p}}|\nabla\phi|)^{p}.\label{lower bdd}
\end{align}

On the other hand, 
\begin{align}
\int_{B(0,2R)}(-\Delta_{p}v_{k})v_{k}^{\lambda-(p-1)}\phi^{p} & =\int_{B(0,2R)}(\frac{-\Delta_{p}v_{k}}{v_{k}^{p-1}})v_{k}^{\lambda}\phi^{p}\nonumber \\
 & \le\|\frac{-\Delta_{p}v_{k}}{v_{k}^{p-1}}\|_{L^\frac{n}{p}(B(0,2R))}\|v_{k}^{\lambda}\phi^{p}\|_{L^\frac{n}{n-p}(B(0,2R))}.\label{upper bdd}
\end{align}
 Notice that 
\begin{align*}
\|\frac{-\Delta_{p}v_{k}}{v_{k}^{p-1}}\|_{L^\frac{n}{p}({B(0,2R)})}^{\frac{n}{p}} & =\int_{B(0,2R)}(-\frac{\Delta_{p}v_{k}}{v_{k}^{p-1}})^{\frac{n}{p}}\\
 & =\int_{B(0,2R)}-\Delta_{p}v_{k}(\frac{-\Delta_{p}v_{k}}{v_{k}^{\frac{n(p-1)}{n-p}}})^{\frac{n-p}{p}}\\
 & \le \tau^\frac{n-p}{p}\int_{B(0,1)}g_{k}\to0
\end{align*}
 as $k\to\infty$. 

Then combining (\ref{lower bdd}) and (\ref{upper bdd}), we find
that there exist a positive integer $k_{0}$ and a positive constant
$C(n,p)$ such that for $k\ge k_{0}$
\[
\|v_{k}^{\lambda}\|_{L^{\frac{n}{n-p}}(B(0,R))}\le\frac{C(n,p)}{R^{p}}\|v_{k}^{\lambda}\|_{L^{1}(B(0,2R))},
\]
or 
\begin{equation}
\|v_{k}\|_{L^{\frac{n\lambda}{n-p}}(B(0,R))}\le(\frac{C(n,p)}{R^{p}})^{\frac{1}{\lambda}}\|v_{k}\|_{L^{\lambda}(B(0,2R))}.\label{iteration}
\end{equation}

We let 
\[
\lambda_{j}=\frac{(n-1)(p-1)}{n-p}(\frac{n}{n-p})^{j},\,\,{\rm for}\,\,j=0,1,2,\cdots.
\]
According to \cite[Proposition 4.3.9]{Veron2017}, we have 
\begin{align*}
\int_{B(0,R)}v_{k}^{\lambda_{0}}(\xi) & =\int_{0}^{\infty}|\{\xi\in B(0,R);v_{k}^{\lambda_{0}}>t\}|dt\\
 & =\int_{0}^{\infty}|\{\xi\in B(0,R);v_{k}>t^\frac{1}{{\lambda_{0}}}\}|dt\\
 & \le c(n,p)\int_{0}^{\infty}\min\{|B(0,R)|,\frac{|\mu_{g_{k}}(B(0,2))|^{\frac{n}{n-p}}}{t^{\frac{n}{n-1}}}\}dt\\
 & \le C(n,p).
\end{align*}
\label{v_k integral estimate}

We fix some $l$ such that $\lambda_{l}>\frac{n^{2}(p-1)}{p(n-p)}$.
Then by iteration $l$ times we have 
\begin{align*}
\|v_{k}\|_{L^{\lambda_{l}}(B(0,\frac{R}{2^{l}}))} & \le(\frac{C(n,p)}{R^{p}})^{\frac{1}{\lambda_{0}}+\cdots+\frac{1}{\lambda_{l-1}}}\|v_{k}\|_{L^{\lambda_{0}}(B(0,R))}\\
 & \le C(n,p,l).
\end{align*}
Then we have 
\begin{align}
\mu_{g_{k}}(B(0,t)) & =\int_{B(0,t)}g_{k}(\xi)d\xi\le C\int_{B(0,t)}v_{k}^{\frac{n(p-1)}{n-p}}(\xi)d\xi\nonumber \\
 & \le(\int_{B(0,t)}(v_{k}^{\frac{n(p-1)}{n-p}})^{\frac{\lambda_{l}}{\frac{n(p-1)}{n-p}}})^{\frac{\frac{n(p-1)}{n-p}}{\lambda_{l}}}|B(0,t)|^{1-\frac{\frac{n(p-1)}{n-p}}{\lambda_{l}}}\nonumber \\
 & \le C(n,p,l)t^{n-p+\epsilon}.\label{mu_g_k estimate}
\end{align}

\end{proof}

\section{Remove the $p$-thin set}\label{Sec4}

In this section, we will prove the following theorem.

\begin{thm}\label{remove the p-thin set}

Suppose $u$ is a solution of (\ref{main differential inequality}).
Then 
\begin{equation}
\lim_{x\to0}\frac{u(x)}{|x|^{\frac{p-n}{p-1}}}=m\ge0,\label{limit is m>0}
\end{equation}
where $m$ is given in Theorem \ref{Theorem 4 of LMQZ2023b}.

\end{thm}

One corollary is 

\begin{cor}

If $m>0$, 
\[
\lim_{x\to0}\frac{u(x)}{\underline{u}(|x|)}=1,
\]where $\underline{u}$ is defined by (\ref{bar u and underline u}).

\end{cor}

The proof of Theorem \ref{remove the p-thin set} also relies on the following lemma.

\begin{lem}\label{close to m}

Under the assumptions in Theorem \ref{main thm}, for any $\varepsilon>0$,
there are $\delta,\alpha_{0}>0$ such that if $0<|y|<\delta$, $\frac{\alpha_{0}}{4}\le\alpha\le\alpha_{0}$
\[
m(1-\varepsilon)\le\frac{\inf_{B(y,\alpha|y|)}u(x)}{|y|^{\frac{p-n}{p-1}}}\le(m+\varepsilon)(1+\varepsilon).
\]

\end{lem}

\begin{proof}$\forall\varepsilon>0$, we let $\alpha_{0}=\min\{1-\frac{1}{(1+\varepsilon)^{\frac{p-1}{n-p}}},\frac{1}{(1-\varepsilon)^{\frac{p-1}{n-p}}}-1\}.$
From \cite[Theorem 4]{Liu-Ma-Qing-Zhong2} there is a set $E$ which
is $p$-thin at $0$ such that 
\[
\lim_{x\to0,x\notin E}\frac{u(x)}{|x|^{\frac{p-n}{p-1}}}=\liminf_{x\to0}\frac{u(x)}{|x|^{\frac{p-n}{p-1}}}=m.
\]
Since $E$ is $p$-thin, we can choose $\delta>0$ such that when
$0<|y|<\delta$, $B(y,\frac{\alpha_{0}}{4}|y|)\not\subset E$. Then
we know that 
\[
\lim_{y\to0}\inf_{B(y,\frac{\alpha_{0}}{4}|y|)}\frac{u(x)}{|x|^{\frac{p-n}{p-1}}}=m.
\]
Then we may choose $\delta$ even smaller, such that when $0<|y|<\delta$,
$\inf_{B(y,\frac{\alpha_{0}}{4}|y|)}\frac{u(x)}{|x|^{\frac{p-n}{p-1}}}\le m+\varepsilon$. 

Let $\inf_{B(y,\alpha|y|)}\frac{u(x)}{|x|^{\frac{p-n}{p-1}}}=\frac{u(x_{0})}{|x_{0}|^{\frac{p-n}{p-1}}}$
for some $x_{0}\in\bar{B}(y,\alpha|y|)$. Then $\frac{|y|}{|x_{0}|}\in[\frac{1}{1+\alpha},\frac{1}{1-\alpha}]$
and 
\[
\frac{\inf_{B(y,\alpha|y|)}u}{|y|^{\frac{p-n}{p-1}}}\le\frac{u(x_{0})}{|y|^{\frac{p-n}{p-1}}}=\frac{u(x_{0})}{|x_{0}|^{\frac{p-n}{p-1}}}\cdot\frac{|x_{0}|^{\frac{p-n}{p-1}}}{|y|^{\frac{p-n}{p-1}}}\le\frac{u(x_{0})}{|x_{0}|^{\frac{p-n}{p-1}}}(\frac{1}{1-\alpha})^{\frac{n-p}{p-1}}\le(m+\varepsilon)(1+\varepsilon).
\]
Next let $\inf_{B(y,\alpha|y|)}u(x)=u(y_{0})$ for some $y_{0}\in\bar{B}(y,\alpha|y|)$.
Then $\frac{|y|}{|y_{0}|}\in[\frac{1}{1+\alpha},\frac{1}{1-\alpha}]$
and 
\[
\frac{\inf_{B(y,\alpha|y|)}u}{|y|^{\frac{p-n}{p-1}}}=\frac{u(y_{0})}{|y|^{\frac{p-n}{p-1}}}=\frac{u(y_{0})}{|y_{0}|^{\frac{p-n}{p-1}}}\frac{|y_{0}|^{\frac{p-n}{p-1}}}{|y|^{\frac{p-n}{p-1}}}\ge m(\frac{1}{1+\alpha})^{\frac{n-p}{p-1}}\ge m(1-\varepsilon).
\]

\end{proof}

\begin{proof}(of Theorem \ref{remove the p-thin set}) Let $\alpha_{0}$ be the constant appearing in Lemma \ref{close to m}. We apply Theorem
\ref{KM theorem} to $u(y)-\inf_{B(y,\frac{3\alpha_{0}}{4}|y|)}u(x)$ and obtain 
\[
u(y)-\inf_{B(y,\frac{3\alpha_{0}}{4}|y|)}u(x)\le c_{2}\inf_{B(y,\frac{\alpha_{0}}{4}|y|)}(u-\inf_{B(y,\frac{3\alpha_{0}}{4}|y|)}u)+c_{2}W_{1,p}^{\mu}(y,\frac{\alpha_{0}}{2}|y|).
\]
Hence 
\begin{align*}
\frac{u(y)}{|y|^{\frac{p-n}{p-1}}}\le & \frac{\inf_{B(y,\frac{3\alpha_{0}}{4}|y|)}u(x)}{|y|^{\frac{p-n}{p-1}}}+c_{2}\frac{\inf_{B(y,\frac{\alpha_{0}}{4}|y|)}u(x)}{|y|^{\frac{p-n}{p-1}}}\\
 & -c_{2}\frac{\inf_{B(y,\frac{3\alpha_{0}}{4}|y|)}u}{|y|^{\frac{p-n}{p-1}}}+c_{2}\frac{W_{1,p}^{\mu}(y,\frac{\alpha_{0}}{2}|y|)}{|y|^{\frac{p-n}{p-1}}},
\end{align*}
which implies, by Lemma \ref{close to m},
\[
\limsup_{y\to0}\frac{u(y)}{|y|^{\frac{p-n}{p-1}}}\le m+\varepsilon C(m)+c_{2}\limsup_{y\to0}\frac{W_{1,p}^{\mu}(y,\frac{\alpha_{0}|y|}{2})}{|y|^{\frac{p-n}{p-1}}}.
\]
On the other hand,
\[
\frac{u(y)}{|y|^{\frac{p-n}{p-1}}}\ge m.
\]
To finish the proof, we just need to show
\[
\lim_{x\to0}\frac{W_{1,p}^{\mu}(x,\frac{1}{2}|x|)}{|x|^{\frac{p-n}{p-1}}}=0.
\]

Suppose it were not true, then there exist a sequence $x_{k}\to0$ and $\delta>0$
such that 
\[
\frac{W_{1,p}^{\mu}(x_{k},\frac{1}{2}|x_{k}|)}{|x_{k}|^{\frac{p-n}{p-1}}}\ge\delta.
\]
Then we define $v_{k},g_{k}$ as in (\ref{v_k  definition})(\ref{g_k definition})
and $\mu_{g_{k}}(\Omega)=\int_{\Omega}g_{k}$. It is easy to calculate
that 
\[
\frac{W_{1,p}^{\mu}(x_{k},\frac{1}{2}|x_{k}|)}{(\frac{|x_{k}|}4)^{\frac{p-n}{p-1}}}=W_{1,p}^{\mu_{g_{k}}}(0,2),\,\,\mu_{g_{k}}(B(0,2))\to0\,\,{\rm as}\,\,k\to\infty.
\]
From (\ref{mu_g_k estimate}), after choosing 
\[
\rho=(\frac{\mu_{g_{k}}(B(0,2))}{C(n,p,l)})^{\frac{1}{n-p+\varepsilon}},
\]
 we calculate

\begin{align*}
W_{1,p}^{\mu_{g_{k}}}(0,2) & =\int_{0}^{2}(\frac{\mu_{g_{k}}(B(0,t))}{t^{n-p}})^{\frac{1}{p-1}}\frac{dt}{t}\\
 & =\int_{0}^{\rho}+\int_{\rho}^{2}(\frac{\mu_{g_{k}}(B(0,t))}{t^{n-p}})^{\frac{1}{p-1}}\frac{dt}{t}\\
 & \le C(n,p,l)\mu_{g_{k}}(B(0,2))^{\frac{\varepsilon}{(n-p+\varepsilon)(p-1)}}\\
 & \to0\,\,{\rm as}\,\,k\to\infty,
\end{align*}
which is a contradiction. Then (\ref{limit is m>0}) is proved.

\end{proof}

\section{Asymptotic behavior of the singularity  when $m=0$}\label{Sec5}

In this section, we will prove the following theorem.

\begin{thm}\label{asymptotic symmetry when m=00003D0}

Suppose $u$ is a $C^{2}$ nonnegative solution of (\ref{main differential inequality})
in $B(0,1)\backslash\{0\}$
and $m=0$. Then 
\begin{equation}
\lim_{x\to0}\frac{u(x)}{\underline{u}(|x|)}=1,\label{limit 1}
\end{equation}where $\underline{u}$ is defined by (\ref{bar u and underline u}).
Moreover, either $u(x)\to\infty$ as $x\to0$ or $u(0)\in(0,\infty)$
and $u$ can be extended to $C^{1,\alpha}(B(0,1))$.

\end{thm}

We will need the following Harnack inequality. 

For $0<\epsilon<1$ and $r>0$, we define 
\[
\Omega_{r}^{\epsilon}=\{y\in\mathbb{R}^{n};\epsilon r\le|y|\le\epsilon^{-1}r\}.
\]

\begin{lem}\label{Harnack inequality}

Under the assumptions of Theorem \ref{main thm} and $m=0$, for any
$0<\epsilon<1$, there holds
\begin{itemize}
\item Either $u(x)\to\infty$ as $x\to0$ and there are positive constants
$\delta(\epsilon),\,\,C$ such that 
\[
\sup_{\Omega_{r}^{\epsilon}}u\le C\inf_{\Omega_{r}^{\epsilon}}u
\]
 for $0<r<\delta(\epsilon)$, 
\item Or $u(0)\in(0,\infty)$ and $u$ can be extended to $C^{1,\alpha}(B(0,1))$. 
\end{itemize}
\end{lem}

We will demonstrate how our result is derived from Serrin's Harnack inequality in \cite{Serrin 64}.

\begin{proof}(of Lemma \ref{Harnack inequality}) First, we may
cover
\[
\Omega_{1}^{\epsilon}=\{y\in\mathbb{R}^{n};\epsilon\le|y|\le\epsilon^{-1}\}
\]
with finitely many balls $B_{i},i=1,\cdots, l$ such that the concentric
balls $3B_{i}\subset\Omega_{1}^{\frac{\epsilon}{2}}$. The number
$l$ depends only on $\epsilon$. 

For $u(y),y\in\Omega_{\delta}^{\frac{\epsilon}{2}}$, we let
\[
v_{\delta}(\xi)=u(\delta\xi),\,\,\forall\xi\in\Omega_{1}^{\frac{\epsilon}{2}}.
\]
Then
\begin{align*}
0\le\frac{-\Delta_{p}^{\xi}v_{\delta}(\xi)}{v_{\delta}(\xi)^{p-1}} & \le\frac{-\delta^{p}\Delta_{p}^{y}u(y)}{u^{p-1}}\le \tau\delta^{p}u^{\frac{p(p-1)}{n-p}}(y)\le\frac{\tau}{|\xi|^{p}}(|y|^{\frac{n-p}{p-1}}u(y))^{\frac{p(p-1)}{n-p}}.
\end{align*}
Since $m=0$, we have
\[
\sup_{y\in\Omega_{\delta}^{\frac{\epsilon}{2}}}|y|^{\frac{n-p}{p-1}}u(y)\to0\,\,{\rm as}\,\,\delta\to0
\]
and therefore 
\[
\sup_{\xi\in\Omega_{1}^{\frac{\epsilon}{2}}}|\frac{-\Delta_{p}^{\xi}v_{\delta}(\xi)}{v_{\delta}(\xi)^{p-1}}|\to0\,\,{\rm as}\,\,\delta\to0.
\]
Then we may apply Serrin's Harnack inequality to $v$ in $B_{i}$
and conclude that 
\[
\sup_{B_{i}}v_{\delta}\le C(\min_{B_{i}}v_{\delta}+k)
\]
for some $k$ depending only on $\sup_{\xi\in\Omega_{1}^{\frac{\epsilon}{2}}}|\frac{-\Delta_{p}^{\xi}v_{\delta}(\xi)}{v_{\delta}(\xi)^{p-1}}|$.
It is now easy to prove
\[
\sup_{\Omega_{1}^{\epsilon}}v_{\delta}\le C(\min_{\Omega_{1}^{\epsilon}}v_{\delta}+k)
\]
and hence 
\begin{equation}
\sup_{\Omega_{\delta}^{\epsilon}}u\le C(\min_{\Omega_{\delta}^{\epsilon}}u+k).\label{harnarck for u}
\end{equation}

If $u(y)\to\infty$ as $y\to0$, one may remove $k$ in the above
inequality. We will prove that, if $u(y)\to\infty$ as $y\to0$ does not
hold, then
\[
\lim_{y\to0}u(y)=u_{0}>0.
\]
In this case, we claim that there can not be a sequence $y_{i}\to0$
such that $u(y_{i})\to\infty$. If there were, notice that we may
find another sequence $\tilde{y}_{i}\to0$ such that $u(\tilde{y}_{i})\le M$.
Then we may find $|y_{i}|<|\tilde{y}_{j}|<|y_{k}|$ such that $u(\tilde{y}_{j})\ll\min\{u(y_{i}),u(y_{k})\}.$
Then from (\ref{harnarck for u}), we have
\[
\sup_{\Omega_{|\tilde{y}_{j}|}^{\epsilon}}u\ll\min\{\inf_{\Omega_{|y_{i}|}^{\epsilon}}u,\inf_{\Omega_{|y_{k}|}^{\epsilon}}u\},
\]
which contradicts the fact that $u$ is $p$-superharmonic. Then
we know that if $u(y)\to\infty$ as $y\to0$ does not hold, then
$u\le M$ near $0$. According to Theorem\ref{Serrin's removable singularity}, we know
$u$ has removable singularity at $0$, which means that  ($\ref{main differential inequality}$)
holds in $B(0,1)$. The reason why we can apply Theorem\ref{Serrin's removable singularity}
is that 
\[
-\Delta_{p}u\le \tau u^{\frac{n(p-1)}{n-p}}\le \tau u^{\frac{p(p-1)}{n-p}}u^{p-1}
\]
and $u^{\frac{p(p-1)}{n-p}}$ is bounded and therefore belongs to $L^{\frac{n}{p-\kappa}}$
for some $\kappa>0$. Then from \cite[Theorem 8]{Serrin 64}, we know
$u$ can be extended to $0$ in $C^{1,\alpha}$ manner for some $\alpha>0$.
It is evident that $u(0)>0$, since $u$ is $p$-superharmonic. 

\end{proof}

Another result we need is the following Liouville-type theorem on
$\mathbb{R}^{n}\backslash\{0\}$. 

\begin{thm}\label{Liouville type thm}

For $1<p<n$, a $p$-harmonic function $u\ge0$ on $\mathbb{R}^{n}\backslash\{0\}$
has the form 
\[
a|x|^{\frac{p-n}{p-1}}+b.
\]

\end{thm}

\begin{rmk}
 This theorem can be regarded as a slightly more general form of \cite[Theorem2.1]{Sciunzi2016}. By using the moving plane method, Sciunzi 
 proved that 
 any  $p$-harmonic function $u$ has the form 
 $a|x|^{\frac{p-n}{p-1}}$for some $a>0$,
 under the assumption of
 \[
 \lim\limits_{x\to0}u(x)=\infty\quad and\quad \lim\limits_{x\to\infty}u(x)=0.
 \]
\end{rmk}

\begin{proof}(of Theorem \ref{Liouville type thm}) From \cite[Corollary, p.84]{Serrin65-2} or \cite[Theorem 2 and Theorem 3]{Avila and Brock},
\[
\lim_{x\to\infty}u(x)=b
\]
exists and 
\[
|u(x)-b|\le C_{1}|x|^{\frac{p-n}{p-1}},x\in\mathbb{R}^{n}\backslash B(0,1).
\]
 From Theorem \ref{Serrin's removable singularity}, eihter
\[
C_{2}^{-1}|x|^{\frac{p-n}{p-1}}\le u(x)\le C_{2}|x|^{\frac{p-n}{p-1}},x\in B(0,1)\backslash\{0\},
\]
or $u$ is bounded near $0$.
Thus, 
\[
0\le u(x)\le C|x|^{\frac{p-n}{p+1}}+C',x\in \mathbb{R}^n\backslash\{0\}.
\]
 Then from \cite[Theorem 2.2]{Kich-Veron}, we finish the proof.

\end{proof}

\begin{proof}(of Theorem \ref{asymptotic symmetry when m=00003D0})

From Lemma \ref{Harnack inequality}, if $u(x)\to\infty$ does not
hold as $x\to0$, $u(x)$ can be extended to $C^{1,\alpha}(B(0,1))$.
Since $u(0)>0$, clearly (\ref{limit 1}) holds.

Now we assume $u(x)\to\infty$ as $x\to0$. Denote $\underline{u}(r)=\inf_{|x|=r}u$.
Define 
\[
u_{r}(\xi)=\frac{u(r\xi)}{\underline{u}(r)},\,\xi\in\Omega_{1}^{\epsilon}.
\]
Direct calculation yields
\[
0\le-\Delta_{p}^{\xi}u_{r}(\xi)=f_{r}(\xi)\le\frac{\tau r^{p}}{\underline{u}(r)^{p-1}}u(r\xi)^{\frac{n(p-1)}{n-p}}.
\]
By the definition of $\underline{u}(r)$ and Lemma \ref{Harnack inequality},
there exists a positive constant $C(\epsilon)$ such that 
\[
C(\epsilon)^{-1}u(r\xi)\le\underline{u}(r)\le C(\epsilon)u(r\xi),\forall\xi\in\Omega_{1}^{\epsilon},
\]
from which we have 
\[
C(\epsilon)^{-1}\le u_{r}(\xi)\le C(\epsilon)
\]
 and 
\begin{align*}
0 & \le f_{r}(\xi)\le C(\epsilon)^{p-1}r^{p}u(r\xi)^{\frac{p(p-1)}{n-p}}=C(\epsilon)^{p-1}r^{p}o(|r\xi|^{-p})\\
 & =C(\epsilon)^{p-1}o(|\xi|^{-p})=C(\epsilon)^{p-1}o(1).
\end{align*}
As we have the $C^{1,\alpha}$ estimate 
\[
\|u_{r}(\xi)\|_{C^{1,\alpha}(\Omega_{1}^{\epsilon})}\le C(\epsilon)
\] from \cite{Serrin 64,Di Benedetto,Tolksdorf}, we deduce that $u_{r}(\xi)$ will converges to some function $\hat{u}(\xi)$
in $C^{1}(\Omega_{1}^{\epsilon})$. Therefore, for any $\phi\in C_{0}^{\infty}(\Omega_{1}^{\epsilon})$,
\begin{align*}
\int_{\Omega_{1}^{\epsilon}}|\nabla_{\xi}\hat{u}(\xi)|^{p-2}\nabla_{\xi}\hat{u}(\xi)\nabla\phi(\xi)d\xi & =\lim_{r\to0}\int_{\Omega_{1}^{\epsilon}}|\nabla_{\xi}u_{r}(\xi)|^{p-2}\nabla_{\xi}u_{r}(\xi)\nabla\phi(\xi)d\xi\\
 & \le\lim_{r\to0}\int_{\Omega_{1}^{\epsilon}}C(\epsilon)^{p-1}o(1)d\xi\\
 & =0,
\end{align*}
which implies that $C(\epsilon)^{-1}\le\hat{u}(\xi)\le C(\epsilon)$ is a
$p$-harmonic function. By letting $\epsilon\to0$ we obtain a nonnegative
$p$-harmonic function $\hat{u}(\xi)$ on $\mathbb{R}^{n}\backslash\{0\}$.
Using Theorem \ref{Liouville type thm}, it follows that
\[
\hat{u}(\xi)=a|\xi|^{\frac{p-n}{p-1}}+b.
\]
Thus, 
\[
\lim_{x\to0}\frac{u(x)}{g(|x|)}=\lim_{x\to0}\frac{u(|x|\frac{x}{|x|})}{g(|x|)}=a+b.
\]
 From the definition of $g$, we know $a+b=1.$ Then the theorem is
proved.

\end{proof}

\section{Precise asymptotic lower bound of $u(x)$, when  $m=0$ and $u(x)\to\infty,$ as $x\to0$}\label{Sec6}

We assume $m=0$ in this section and moreover assume $u(x)\to\infty$
as $x\to0$. We will prove the following theorem.

\begin{thm}\label{lower log bound thm}

Suppose $u$ is a $C^{2}$ nonnegative solution of (\ref{main differential inequality})
in $B(0,1)\backslash\{0\}$. Assume 
\[
\lim_{x\to0}\frac{u(x)}{|x|^{\frac{p-n}{p-1}}}=0
\]
 and $u(x)\to\infty$ as $x\to0$, then 
\[
\liminf_{r\to0^+}(\log\frac{1}{r})^{\frac{n-p}{p(p-1)}}r^{\frac{n-p}{p-1}}u(r)\ge (\frac{1}{\tau}(\frac{n-p}{p})(\frac{n-p}{p-1})^{p-1})^\frac{n-p}{p(p-1)}.
\]

\end{thm}

\begin{Def}
\begin{itemize}
\item A function $f(r),r\in(\alpha,\beta)$ is called convex if for any
$r_{1},r_{2}\in(\alpha,\beta)$, there holds
\[
f(\lambda r_{1}+(1-\lambda)r_{2})\le\lambda f(r_{1})+(1-\lambda)f(r_{2}),\forall\lambda\in(0,1).
\]
\item A function $f(r),r\in(\alpha,\beta)$ is called semi-convex if there
exists $C>0$ such that 
\[
f(r)+Cr^{2}
\]
 is convex. 
\item A function $f(r),r\in(\alpha,\beta)$ is called locally semi-convex
if for any $[\alpha_{1},\beta_{1}]\subset(\alpha,\beta)$, $f$ is
semi-convex in $(\alpha_{1},\beta_{1})$. 
\item If $-f$ is convex, semi-convex or locally semi-convex, we say $f$
is concave, semi-concave or locally semi-concave.
\end{itemize}
\end{Def}

For the properties of locally semi-convex functions, we have the next lemma. 

\begin{lem}

If $f(r)$, $r\in(\alpha,\beta)$ is locally semi-convex, then 
\begin{itemize}
\item $f$ is continuous.
\item $f_{r}^{-},f_{r}^{+}$ exists everywhere. And $f_{r}^{-}(r)\le f_{r}^{+}(r)$,
where ``$<$'' holds only on at most countably many point $\{r_{i}\}$,
and on $(\alpha,\beta)\backslash\{r_{i}\}$, $f_{r}^{\pm}$ is continuous. 
\item $f_{rr}(r)$ exists almost everywhere. There is a nonnegative Radon measure $\nu$ such that for any $\phi\in C_{0}^{\infty}((\alpha,\beta))$
\[
-\int_{\alpha}^{\beta}f\phi_{rr}dr=\int_{\alpha}^{\beta}\phi d\nu.
\]
Moreover, $\nu$ has the following decomposition into three Radon measures,
\[
\nu=\nu_{ac}+\nu_{pp}+\nu_{sc},.
\]
where $\nu_{ac}$ is absolutely continuous with respect to Lebesgue
measure, $\nu_{pp}$ is the pure point component with support $\{r_i\}$, $\nu_{sc}$ is singular
continuous component and $\nu_{pp},\nu_{sc}$ are nonnegative. For
short we write 
\[
f_{rr}(r)=\nu.
\]
\item The generalized Newton-Leibniz formula holds: for any $\alpha<\alpha_{1}<\beta_{1}<\beta$
\[
f_{r}^{-}(\beta_{1})-f_{r}^{+}(\alpha_{1})=\nu((\alpha_{1},\beta_{1})).
\]
\end{itemize}
\end{lem}

The proof is a standard exercise in real analysis, see e.g. \cite[p.157]{Rudin1987} or \cite[p.333]{Ziemer2017}.

By rescaling, we can assume that $\tau=1$ in (\ref{main differential inequality}) without loss of generality. Now suppose a $C^{2}(B(0,1)\backslash\{0\})$ function $u\ge0$
satisfies 
\begin{equation}
0\le-\Delta_{p}u\le u^{\frac{n(p-1)}{n-p}},x\in B(0,1)\backslash\{0\}.\label{main inequality}
\end{equation}

Let $\eta=\frac{n-p}{p-1}$, $t=\eta\log\frac{\eta}{r}\in(\eta\log\eta,\infty)$
and 
\[e^{t}\bar{w}(t)=\bar{u}(r),\,\,e^{t}\underline{w}(t)=\underline{u}(r)\]
(recall that $\bar{u}(r),\underline{u}(r)$ are defined as $\underline{u}(r)=\inf\limits_{\partial B(0,r)}u(x),\,\,\bar{u}(r)=\sup\limits_{\partial B(0,r)}u(x)$).

\begin{lem}\label{properties of l.s.c. functions}
\begin{enumerate}
\item $\underline{w}(t)$ is locally semi-concave in $(\eta\log\eta,\infty)$
and 
\[
\lim_{t\to\infty}\underline{w}(t)=0.
\]
\item $\bar{w}(t)$ is locally semi-convex in $(\eta\log\eta,\infty)$
\[
\lim_{t\to\infty}\bar{w}(t)=0.
\]
We denote
\[
\bar{w}_{tt}(t)=\mu=\mu_{ac}+\mu_{pp}+\mu_{sc}
\]
 with $\mu_{pp},\mu_{sc}$ nonnegative. The support of $\mu_{sc}$
has $0$ Lebesgue measure and $\mu_{pp}$ is supported on $\{\tilde{t}_{i}\}$. 
\item For any $k>0$ 
\[
e^{kt}\underline{w}(t),\,\,e^{kt}\bar{w}(t)\to\infty,\,\,{\rm as}\,\,t\to\infty.
\]
\end{enumerate}
\end{lem}
\begin{proof}
First we prove that $\underline{u}$ is locally semi-concave and $\bar{u}$
is locally semi-convex. For $[r_{0},r_{1}]\subset(0,1)$, we may assume
that 
\[
\sup_{|x|\in[r_{0},r_{1}]}|\nabla^{2}u|\le M(r_0,r_1)<\infty.
\]
It is apparent that  $u+\frac{M}{2}|x|^{2}$ is convex in $\{x;|x|\in[r_{0},r_{1}]\}.$
From this, we deduce that 
\[
\bar{u}(r)+\frac{M}{2}r^{2}=\sup_{|x|=r}(u+\frac{M}{2}|x|^{2})
\]
is convex in $r\in(r_{0},r_{1})$, which means $\bar{u}(r)$
is locally semi-convex. Similarly,  $\underline{u}$ is
locally semi-concave. Then it is easy to obtain that
$\bar{w}(\underline{w})$ is locally semi-convex (locally semi-concave).

Since $\underline{w}(t)=(\frac{r}{\eta})^{\eta}\underline{u}(r)$
and $(\frac{r}{\eta})^{\eta}\underline{u}(r)\to0$ as $r\to0$,
we see $\underline{w}(t)\to0$ as $t\to\infty.$ Since $\bar{u}(r)/\underline{u}(r)\to1,$
as $r\to0$, we know $\bar{w}(t)\to0$ as $t\to\infty.$

To prove Item 3, we assume if $\limsup_{t\to\infty}e^{kt}\underline{w}(t)$
were finite, there would be $c>0$ such that 
\[
\limsup_{x\to0}\frac{u(x)}{|x|^{(1-k)\frac{p-n}{p-1}}}=c.
\]
It follows that $u\in L^{(1+\theta)\frac{n(p-1)}{n-p}}$ for some
$\theta(k)>0$. Thus, $u$ has a removable singularity at $0$, which contradicts the fact that
as $x\to0$, $u\to\infty$(see\cite[Theorem10]{Serrin 64}). Then for any $k>0$, we have $e^{kt}\bar{w}(t),e^{kt}\underline{w}(t)\to\infty$
as $t\to\infty.$ 

\end{proof}

We denote $\bar{z}(t)=\bar{w}_{t}^{-}(t)$ and $q=\frac{n(p-1)}{n-p}$, and consider $(\bar{w}(t),\bar{z}(t))$
on $(\bar{w},\bar{z})$-plane. 

\begin{lem}\label{jump above}
\begin{enumerate}
\item $\bar{w}(t)$ is continuous for $t\ge\eta\log\eta$ and $\bar{z}(t)$
is continuous for $t\ge\eta\log\eta, t\ne\tilde{t}_{i}$.
For $t\in supp(\mu_{pp})=\{\tilde{t}_{i}\},$ $\bar{w}_{t}^{-}(\tilde{t}_{i})<\bar{w}_{t}^{+}(\tilde{t}_{i})$.
\item If $\bar{w}(t)+\bar{z}(t)>0$ for $t\in[t_{0},t_{1}]\subset(\eta\log\eta,\infty)$, then
\[
\bar{z}(t_{1})-\bar{z}(t_{0})\ge\bar{w}_{t}^{-}(t_{1})-\bar{w}_{t}^{+}(t_{0})\ge\int_{t_{0}}^{t_{1}}(-\bar{z}-\frac{\bar{w}^{q}}{(p-1)(\bar{w}+\bar{z})^{p-2}})dt.
\]
\end{enumerate}
\end{lem}

\begin{proof}

The first item is clear, and we will prove the second one. 

Suppose $\bar{u}_{rr}(r)$ exists at a point $r_{0}$. Since $\bar{u}(|x|)$
touches $u(x)$ from above at some point $|x_{0}|=r_{0}$, we have 
\[
-(\Delta_{p}\bar{u})|_{|x|=r_{0}}\le u^q(x_{0})=\bar{u}^q(r_{0}).
\]
 Once $\bar{w}(t_0)+\bar{z}(t_0)>0$  at $t_{0}=\eta\log\frac{\eta}{r_{0}}$, after direct calculation we obtain
\[
-\frac{d}{dt}(\bar{w}+\bar{w}_{t})^{p-1}|_{t_0 }\le\bar{w}^{q}(t_0).
\]
 Then as long as $\bar{w}$ has second derivatives and $\bar{w}+\bar{w}_{t}>0$, we can prove that 
\[
\bar{z}_{t}\ge-\bar{z}-\frac{\bar{w}^{q}}{(p-1)(\bar{w}+\bar{z})^{p-2}}.
\]
Since $\mu_{pp},\mu_{sc}$ are nonnegative, we conclude the proof of  Item 2. 

\end{proof}
\begin{lem}\label{not below a line}

For any $K>0$, there does not exist $t_{0}>0$ such that the inequality
$\bar{z}(t)\leq -K\bar{w}(t)$ holds for all $t\in[t_{0},\infty)$.

\end{lem}

\begin{proof}

If $\bar{z}(t)\le -K\bar{w}(t)$ holds for some $K>0$ and all $t\in[t_{0},\infty)$,
then $\bar{z}(t)<0$ and $\bar{w}(t)$ monotonically decreases. Then
\[
\frac{d\bar{w}(t)}{dt}\le -K\bar{w}(t).
\]
Then $dt\le-\frac{d\bar{w}(t)}{K\bar{w}(t)}$, which implies for any
$t_{1}>t_{0}$
\begin{align*}
t_{1}-t_{0} & \le-\frac{1}{K}\int_{t_{0}}^{t_{1}}\frac{d\bar{w}(t)}{\bar{w}(t)} =-\frac{1}{K}\log\bar{w}(t)|_{t_{0}}^{t_{1}}\\
 & =-\frac{1}{K}(\log\bar{w}(t_{1})-\log\bar{w}(t_{0})).
\end{align*}
Then 
\[
\bar{w}(t_{1})\le\bar{w}(t_{0})e^{-K(t_{1}-t_{0})}
\]
which violates Item 3 of Lemma \ref{properties of l.s.c. functions}.  
\end{proof}

We let
\[
\Omega_{k}^{\delta}=\{(\bar{w},\bar{z});0<\bar{w}\le(\frac{\delta}{k})^{\frac{n-p}{p(p-1)}},\,\,\bar{z}>-k\bar{w}^{q-p+2}\}.
\]
The following lemma is the key of our argument.

\begin{lem}\label{stay in Omega}

For fixed $0<\varepsilon <1$ small enough, we let $k=k(\varepsilon)=\frac{1}{(p-1)(1-\varepsilon)}$ and 
\[0\le\delta=\delta(\varepsilon)=\min\{\frac{1}{2^p|p-2|},\frac{1}{4(q-p+2)}\}\varepsilon\le\frac{1}{2}.\]
There is $T>0,$ such that if  $t_0>T$ and $(\bar{w}(t_{0}),\bar{z}(t_{0}))\in\Omega_{k}^{\delta}$, then for $t>t_{0}$, $(\bar{w}(t),\bar{z}(t))\in\Omega_{k}^{\delta}$. 

\end{lem}

\begin{proof}

Since 
\[
\lim_{t\to\infty}\bar{w}(t)=0,
\]
we may first choose $T$ such that when $t\ge T$, $0<\bar{w}(t)<(\frac{\delta}{k})^{\frac{n-p}{p(p-1)}}.$
If in addition for some $t_{0}\ge T$, $(\bar{w}(t_{0}),\bar{z}(t_{0}))\in\Omega_{k}^{\delta}$,
from the first item of Lemma \ref{jump above}, that the only reason why $(\bar{w}(t),\bar{z}(t))\in\Omega_{k}^{\delta}$
may fail for $t\ge t_{0}$ is that there is $t_{1}>t_{0}$ such that
$\bar{z}(t_{1})=-k\bar{w}^{q-p+2}(t_{1})$ and for $t\in[t_{0},t_{1})$,
$(\bar{w}(t),\bar{z}(t))\in\Omega_{k}^{\delta}$ and when $t\to t_{1}^{-}$,
$(\bar{w}(t),\bar{z}(t))\to(\bar{w}(t_{1}),\bar{z}(t_{1}))$. So we
have 

\begin{align}
\liminf_{t\to t_{1}^{-}}\frac{\bar{z}(t)-\bar{z}(t_{1})}{\bar{w}(t)-\bar{w}(t_{1})} & \ge\frac{d(-k\bar{w}^{q-p+2})}{d\bar{w}}|_{\bar{w}=\bar{w}(t_{1})}\nonumber \\
 & =-k(q-p+2)\bar{w}^{q-p+1}\nonumber \\
 & \ge-(q-p+2)k(\frac{\delta}{k})^{\frac{n-p}{p(p-1)}(q-p+1)}\nonumber \\
 & =-(q-p+2)\delta\nonumber \\
 & \ge-\frac{\varepsilon}{4}.\label{derivative lower bdd}
\end{align}
From Lemma \ref{properties of l.s.c. functions}, we know 
\[
\bar{w}_{tt}=\nu_{ac}+\nu_{pp}+\nu_{sc}
\]
in which $\nu_{pp}$ and $\nu_{sc}$ are nonnegative. Since 
\[
\bar{z}_{t}\ge-\bar{z}-\frac{\bar{w}^{q}}{(p-1)|\bar{w}+\bar{z}|^{p-2}}
\]
in the sense of distribution we know for $t_{2}<t_{1}$
\begin{align*}
\bar{z}(t_{1})-\bar{z}(t_{2}) & =\int_{t_{2}}^{t_{1}}d\bar{z}\\
 & \ge\int_{t_{2}}^{t_{1}}(-\bar{z}-\frac{\bar{w}^{q}}{(p-1)|\bar{w}+\bar{z}|^{p-2}})dt.
\end{align*}
Since we may assume $\bar{w}(t_{1})<\bar{w}(t_{2})$ as $\bar{z}(t)<0$
when $t\in[t_{2},t_{1}]$, we see 
\begin{align*}
\limsup_{t_{2}\to t_{1}^{-}}\frac{\bar{z}(t_{1})-\bar{z}(t_{2})}{\bar{w}(t_{1})-\bar{w}(t_{2})} & =\limsup_{t_{2}\to t_{1}^{-}}\frac{\bar{z}(t_{1})-\bar{z}(t_{2})}{t_{1}-t_{2}}\frac{t_{1}-t_{2}}{\bar{w}(t_{1})-\bar{w}(t_{2})}\\
 & \le\limsup_{t_{2}\to t_{1}^{-}}\frac{\int_{t_{2}}^{t_{1}}(-\bar{z}-\frac{\bar{w}^{q}}{(p-1)|\bar{w}+\bar{z}|^{p-2}})dt}{t_{1}-t_{2}}\frac{t_{1}-t_{2}}{\bar{w}(t_{1})-\bar{w}(t_{2})}\\
 & =(-\bar{z}(t_{1})-\frac{\bar{w}^{q}(t_{1})}{(p-1)|\bar{w}(t_{1})+\bar{z}(t_{1})|^{p-2}})\frac{1}{\bar{z}(t_{1})}\\
 & =-1+\frac{\bar{w}^{p-2}(t_{1})}{k(p-1)|\bar{w}(t_{1})+\bar{z}(t_{1})|^{p-2}}\\
 & =-1+\frac{1-\varepsilon}{|1-k\bar{w}^{q-p+1}(\beta)|^{p-2}}\\
 & \le-1+\frac{1}{(1-\delta)^{p-2}}(1-\varepsilon)\\
 & \le-1+\max\{1,1+2^{p-1}(p-2)\delta\}(1-\varepsilon)\\
 & \le-\frac{\varepsilon}{2},
\end{align*}
which contradicts with (\ref{derivative lower bdd}). Here we used the fact that
\[
\frac{1}{(1-\delta)^{p-2}}\le 1+2^{p-1}(p-2)\delta
\]
for $p\ge2$ and $0\le \delta \le\frac{1}{2}$.
\end{proof}

Now Theorem \ref{lower log bound thm} will follow from the following
lemma.

\begin{lem} \label{exist in Omega}

For fixed $0<\varepsilon<1$ small enough,   $k(\varepsilon),\delta(\varepsilon),T$
given in Lemma \ref{stay in Omega}, there is $t_{0}\ge T$ such that
$(\bar{w}(t_{0}),\bar{z}(t_{0}))\in\Omega_{k}^{\delta}$.

\end{lem}

\begin{proof}

We may assume for $t$ large $\bar{z}(t)<0$ or we will find $t_{0}>T$
and $\bar{z}(t_{0})\ge0>-k\bar{w}(t_{0})^{q-p+2}.$ Then we may assume
$\bar{w}(t)$ is monotonically decrasing for $t$ large. 

Notice that $((\frac{\delta}{k})^{\frac{n-p}{p(p-1)}},-k(\frac{\delta}{k})^{\frac{q-p+2}{q-p+1}})$
lies on both $\bar{z}=-k\bar{w}^{q-p+2}$ and $\bar{z}=-\delta\bar{w}$. From Lemma \ref{not below a line}, we know that there must be some
$t_{0}$ large such that $\bar{z}(t_{0})>-k(\varepsilon)\bar{w}(t_{0}).$
Then there exist some $\varepsilon_{0}>\varepsilon,$ such that
$\bar{w}(t_{0})=(\frac{\delta(\varepsilon_{0})}{k(\varepsilon_{0})})^{\frac{n-p}{p(p-1)}}$
and $\bar{z}(t_{0})>-k(\varepsilon)(\frac{\delta(\varepsilon_{0})}{k(\varepsilon_{0})})^{\frac{q-p+2}{q-p+1}}.$
We may choose some $\varepsilon_{1}$ slightly bigger such that 
\[
(\bar{w}(t_{0}),\bar{z}(t_{0}))\in\Omega_{k(\varepsilon_{1})}^{\delta(\varepsilon_{1})}.
\]
Then we know that for $t\ge t_{0}$, $(\bar{w}(t_{1}),\bar{z}(t_{1}))\in\Omega_{k(\varepsilon_{1})}^{\delta(\varepsilon_{1})}.$ 

If $(\bar{w}(t),\bar{z}(t))$ always lies between $\bar{z}=-k(\varepsilon_{1})\bar{w}(t)^{q-p+2}$
and $\bar{z}=-k(\varepsilon)\bar{w}(t)^{q-p+2}$ for $t$ large. It
is obvious that 
\[
\frac{d\bar{z}}{d\bar{w}}(t_{i})\to0
\]
 for any sequence $t_{i}\to\infty$, where $\frac{d\bar{z}}{d\bar{w}}$
is defined. 

However, on the other hand 
\[
\frac{d\bar{z}}{d\bar{w}}(t_{i})\le-1+\frac{1-\varepsilon}{|1-k_{1}\bar{w}(t_{i})^{q-p+1}|^{p-2}}\le-\frac{\varepsilon}{2},
\]
which is a contradiction. Then we proved the lemma.

\end{proof}

\section{Proof of Theorem \ref{strengthened thm}}\label{Sec7}

In this section, we prove Theorem \ref{strengthened thm}.  First we look at one example. By direct computation we know, for a positive integer $Q$, $$u(x)=\frac{1}{|x|^{\frac{n-p}{p-1}}\log_{Q+2}\frac{1}{|x|}},$$ satisfies\begin{equation}\label{perturbation ineq}
0\le-\Delta_{p}u\le \frac{u^{\frac{n(p-1)}{n-p}}}{(\log_1u)(\log_2 u)\cdots(\log_{Q-1}u)(\log_Q u)}.
\end{equation}
Then we know the assumption of Theorem \ref{strengthened thm} is nearly optimal.

\begin{proof}[Proof of Theorem \ref{strengthened thm}]
   We follow the argument in Section $4$ in \cite{Taliaferro06}. 
   
   Following the notations in Section \ref{Sec6}, we have
   \begin{equation}
\bar{z}+\bar{z}_{t}\ge-\frac{\bar{w}^{q}}{(p-1)(\bar{w}+\bar{z})^{p-2}}\cdot\frac{1}{(\log_{1}e^t\bar{w})(\log_{2}e^t\bar{w})\cdots(\log_{Q-1}e^t\bar{w})(\log_{Q}e^t\bar{w})^\beta},
   \end{equation}
as long as $\bar{w}$ has second derivatives and $\bar{w}+\bar{w}_{t}>0$. 
Assume otherwise, $u$ satisfies Item $3$ of Theorem \ref{main thm} with $\tau=1$, i.e. 

$$\lim\limits_{t\to\infty}\bar{w}(t)=0$$ and 
\[
\liminf_{t\to\infty}\bar{w}(t)\ge (\frac{n-p}{p}\frac{1}{t})^{\frac{1}{q+1-p}},
\]
from which we have, for $t$ large enough,
\begin{equation}\label{lower bound of bar w in section 7}
   \bar{w}(t)\ge \frac{1}{2}(\frac{n-p}{p}\frac{1}{t})^{\frac{1}{q+1-p}}.
\end{equation}
Since 
\[
\log_j(e^t\bar{w}(t))=(\log_{j-1}t)(1+o(1))\,\,\, \text{as}\,\,\, t\to \infty,
\]
we obtain
\begin{equation}\label{lower bound of z and z_t in section 7}
    \bar{z}(t)+\bar{z}_{t}(t)\ge-\frac{\bar{w}(t)^{q}g(t)}{(p-1)(\bar{w}(t)+\bar{z}(t))^{p-2}},
\end{equation}
where 
\[
g(t)=\frac{2}{t(\log_1 t)\cdots(\log_{Q-2}t)(\log_{Q-1}t)^\beta}.
\]
Notice that   
\[
\int_{t_0}^T \bar{z}(t)e^tdt=\bar{z}(T)e^T-\bar{z}(t_0)e^{t_0}-\int_{t_0}^{T}e^td\mu,
\]
where $T>t_0>0$ large enough and $\mu=\bar{w}_{tt}$ is defined in Lemma \ref{properties of l.s.c. functions}.
Similar to the proof of Item $2$ in  Lemma \ref{jump above}, we have 
\begin{equation}\label{integration by part of e^t bar z }
    \begin{aligned}
    e^T\bar{z}(T)-e^{t_0}\bar{z}(t_0)&=
      \int_{t_0}^T e^t\bar{z}(t)dt+\int_{t_0}^T e^td\mu\\
      &\ge-\int_{t_0}^Te^t\frac{\bar{w}(t)^{q}g(t)}{(p-1)(\bar{w}(t)+\bar{z}(t))^{p-2}} dt.
    \end{aligned}
\end{equation}
Denoting
\[
I(T):=\int_{t_0}^T e^t\frac{\bar{w}(t)^{q}g(t)}{(p-1)(\bar{w}(t)+\bar{z}(t))^{p-2}}dt,
\]
from $(25)$ we have 
\begin{equation}\label{upper bound of -bar z(T) in section 7 with I(t)}
    -\bar{z}(T)\le e^{-T}I(T)+Ce^{-T} \,\,\,\text{for}\,\,\, T\ge t_0.
\end{equation}
By Lemma \ref{stay in Omega} and Lemma \ref{exist in Omega},
\begin{equation}\label{upper bound of I(t)}
    \begin{aligned}
        I(T)&=\int_{t_0}^T e^t\frac{\bar{w}(t)^{q}g(t)}{(p-1)(\bar{w}(t)+\bar{z}(t))^{p-2}}dt\\
        &\le\int_{t_0}^T e^t\frac{\bar{w}(t)^{q}g(t)}{(p-1)(\bar{w}(t)-\frac{1}{p-1}\bar{w}(t)^{q-p+2})^{p-2}}dt\\
        &\le \frac{2}{p-1}\int_{t_0}^Te^t\bar{w}(t)^{q-p+2}g(t)dt\\
        &= \frac{2}{p-2}\bigg\{[e^t\bar{w}(t)^{q-p+2}g(t)]\bigg|^T_{t_0}+J(T)\bigg\}
    \end{aligned}
\end{equation}
for $t_0>0$ large enough, where 
\[
J(T)=-\int_{t_0}^T e^td(\bar{w}(t)^{q-p+2}g(t)).
\]

The definition of $J(T)$  makes sense since it is easy to verify that $F(t)=\bar{w}(t)^{q-p+2}g(t)$ 
is locally semi-convex, as $g(t)$ is a smooth function.
Moreover, there holds
\[
F_t(t)= \bar{w}(t)^{q-p+2}g_t(t)+ (q-p+2)\bar{w}(t)^{q-p+1}\bar{z}(t)g(t).
\]
Thus,  
\begin{equation}
    J(T)=-\int_{t_0}^Te^tdF(t)=-\int_{t_0}^Te^t\bar{w}(t)^{q-p+2}g(t)\frac{g_t(t)}{g(t)}dt
-(q-p+2)\int_{t_0}^Te^t\bar{w}(t)^{q-p+2}\frac{\bar{z}(t)}{\bar{w}(t)}g(t)dt.
    \end{equation}
Noticing that $\frac{g'(t)}{g(t)}=O(\frac{1}{t})$ while $\frac{\bar{z}(t)}{\bar{w}(t)}=o(1)$ as $t\to\infty$, we obtain
\[
J(T)\le \frac{1}{2}\int_{t_0}^T e^t\bar{w}(t)^{q-p+2}g(t)dt.
\]
Hence, from  (\ref{upper bound of I(t)}), there exists $t_1>t_0$ such that 
\[
I(T)\le \frac{8}{p-1}e^T\bar{w}(T)^{q-p+2}g(T) \,\,\,\text{for}\,\,\, T\ge t_1.
\]
By (\ref{upper bound of -bar z(T) in section 7 with I(t)}), we have 
\begin{equation}\label{upper bound of bar z(T) in section 7 without I(t)}
    -\bar{z}(T)\le \frac{16}{p-1}g(T)\bar{w}(T)^{q-p+2}
\end{equation}
for $T\ge t_1$, by increasing $t_1$ if necessary. After multiplying (\ref{upper bound of bar z(T) in section 7 without I(t)}) by $\bar{w}^{-(q-p+2)}$ and integrating from $t_1$ to $T$ we have 
\begin{equation}
    \infty\leftarrow \frac{1}{q-p+1}(\frac{1}{\bar{w}(T)^{q-p+1}}-\frac{1}{\bar{w}(t_1)^{q-p+1}})\le \frac{16}{p-1}\int_{t_1}^\infty g(t)dt<\infty,
\end{equation}
which is a contradiction. Then we end the proof.
\end{proof}

\vskip 0.3cm

\noindent Shiguang Ma: School of Mathematical Science and LPMC, Nankai University, Tianjin, China; \\e-mail: 
msgdyx8741@nankai.edu.cn \\
The author is supported by  Natural Science Foundation of Tianjin (Grant No.
22JCJQJC00130), the Fundamental Research Funds for the Central Universities, and Natural Science Foundation of Fujian Province, China (Grant No. 2022J02050).

\vspace{0.2cm}

\noindent Shengyang Zang: School of Mathematical Science, Nankai University, Tianjin, China;
e-mail: 1120230032@mail.nankai.edu.cn\\

\end{document}